\documentclass[11pt]{article}

\usepackage{a4wide}

\usepackage{amsmath,amsbsy,amssymb,latexsym}
\usepackage{amssymb,amsthm}
\usepackage{url}

\usepackage{natbib}
\usepackage{setspace}

\newtheorem{theorem}{Theorem}[section]

\newtheorem{lemma}[theorem]{Lemma}
\newtheorem{definition}[theorem]{Definition}
\newtheorem{remark}{Remark}

\newcommand{\R}{\mathbb{R}}

\begin{document}

\title{On Fractional Variational Problems\\ which Admit Local Transformations}

\author{Agnieszka B. Malinowska\\
\url{a.malinowska@pb.edu.pl}}

\date{Faculty of Computer Science,
Bia{\l}ystok University of Technology,\\
15-351 Bia\l ystok, Poland}

\maketitle\


\begin{abstract}
We extend the second Noether theorem to fractional variational problems which are invariant under infinitesimal transformations that depend upon
$r$ arbitrary functions and their fractional derivatives in the sense of Caputo. Our main result is illustrated using the fractional Lagrangian density of the electromagnetic field.

\bigskip


\medskip


\medskip

\end{abstract}


\section{Introduction}

In 1918 Emmy Noether published a paper that is now famous for Noether's theorem. But in fact Noether proved two theorems in the 1918 paper.
The first theorem explains the correspondence between conserved quantities and continuous symmetry transformations that depend on constant parameters. Such transformations are global transformations. Familiar examples from classical mechanics include the connections between:
spatial translations and conservation of linear momentum; spatial rotations and conservation of angular momentum; and time translations and conservation
of energy (\citet{gelfand,logan2}). The second theorem, less well known, guarantees syzygies between the Euler--Lagrange equations for a variational problem which is invariant under transformations that depend on arbitrary functions and their derivatives. Such transformations are local transformations. The statement of the theorem for this case is very general and, aside from its application to general relativity, it applies in a wide variety of other cases. For example, quantum chromodynamics and other gauge field theories are theories to which it applies. From the second theorem, one has identities between Lagrange expressions and their derivatives. These identities Noether called ``dependencies''. For example, the Bianchi identities, in the general theory of relativity, are examples of such ``dependencies''. In electrodynamics, if the Lagrangian represents a charged particle interacting with a electromagnetic field, one finds that it is invariant under the combined action of the so called gauge transformation of the first
kind on the charged particle field, and a gauge transformation of the second kind on the electromagnetic field. As a result of this invariance it follows, from second Noether's theorem, the conservation of charge. For a complete history of Noether's two theorems on variational symmetries see \citet{brading1,Kosmann}; and also \citet{Carinena,hydon,logan1,torres} for some other generalizations.

Fractional calculus is a discipline that studies integrals
and derivatives of non-integer (real or complex) order
(\citet{book:Kilbas,book:Klimek,book:Podlubny,samko}).
The field was born in 1695 and became an ongoing topic
with many well-known mathematicians contributing to its theory (see \citet{TM:old} for a review).
Fractional derivatives are nonlocal operators and are historically applied
in the study of nonlocal or time dependent processes. The first and well established application
of fractional calculus in Physics was in the framework of
anomalous diffusion, which are related to features observed in many physical
systems, e.g. in dispersive transport in amorphous semiconductor, liquid crystals,
polymers, proteins, etc. Electromagnetic equations (\citet{Baleanu3}) and
field theories (\citet{MyID:182,Cresson,herman,Tarasov}) have been
also recently considered in the context of fractional calculus.
The subject is nowadays very active due to its many applications
in mechanics, chemistry, biology, economics, and control theory (see \citet{book:Baleanu,ferreira,hilfer2,TM:rec} for a review).

One of the most remarkable applications of fractional calculus is in the context of
classical mechanics. \citet{rie,rie2} showed that a Lagrangian involving fractional
time derivatives leads to an equation of motion with nonconservative forces such
as friction. It is a remarkable result since frictional and nonconservative forces
are beyond the usual macroscopic variational treatment, and consequently,
beyond the most advanced methods of classical mechanics. Riewe generalized
the usual variational calculus, by considering Lagrangians that dependent on fractional derivatives,
in order to deal with nonconservative forces. Recently, several
approaches have been developed to generalize the least action principle and the
Euler--Lagrange equations to include fractional derivatives. Investigations cover problems depending on the Caputo fractional derivatives (\citet{agrawalCap,Shakoor:01,Baleanu:Agrawal,BALEANU,MyID:169,MyID:217,Tarasov2}),
the Riemann--Liouville fractional derivatives (\citet{Almeida2,Atanackovic,Baleanu:Avkar,Baleanu:MUSLIH,Herzallah:Baleanu}) and others (\citet{Agrawal,El-Nabulsi:Torres07,El-Nabulsi,p182:Jumarie4,mal}).

Noether's first theorems have been extended to fractional variational problems using several approaches
(\citet{Atanackovic2,Cresson,MyID:068,Frederico:Torres08,MyID:149}). To the best of our knowledge, no second Noether type theorem
is available for the fractional setting. Such a generalization is the aim of this paper. We prove the fractional second Noether theorem for one and multiple dimensional Lagrangians, and we show how our results can be applied for fractional electromagnetic field. We trust that this paper will open several new directions of research and applications. Although we choose the
Caputo calculus, our results are general and can be straightforward
generalized to other fractional calculus, like the Riemann--Liouville fractional calculus
and others approaches.

The paper is organized in the following way. In Section~\ref{sec2} we recall the notion of fractional derivatives and their basic properties, that are needed in the sequel. The intended fractional second Noether type theorem is formulated and proved in Section~\ref{main}, for single (Subsection~\ref{main:single}) and multiple (Subsection~\ref{main:multiple}) integral problems. Our main result is illustrated in Section~\ref{example} using the fractional Lagrangian density of the electromagnetic field.

\section{Fractional Derivatives}
\label{sec2}

In this section we review the necessary definitions and facts from
the fractional calculus. For more on the subject we refer the reader to
(\citet{book:Kilbas,book:Klimek,book:Podlubny,samko}).

Let  $\alpha \in \R$ and $0<\alpha<1$, $f\in L_1([a,b],\R)$. By the left Riemann--Liouville fractional integral of $f$ on the interval $[a,b]$ we mean a function ${_aI_x^\alpha}f$ defined by:
\begin{equation}\label{RLFI1}
{_aI_x^\alpha}f(x)=\frac{1}{\Gamma(\alpha)}\int_a^x
(x-t)^{\alpha-1}f(t)dt,\quad x\in[a,b] \mbox{ a.e. }
\end{equation}
By the right Riemann--Liouville fractional integral of $f$ on the interval $[a,b]$ we mean a function ${_xI_b^\alpha}f$ defined by:
\begin{equation}\label{RLFI2}
{_xI_b^\alpha}f(x)=\frac{1}{\Gamma(\alpha)}\int_x^b(t-x)^{\alpha-1}
f(t)dt,\quad x\in[a,b] \mbox{ a.e. }
\end{equation}
$\Gamma(\cdot)$ represents the Gamma function. For $\alpha=0$, we set
${_aI_x^0}f={_xI_b^0}f:=If$, the identity operator.

If the function ${_aI_x^\alpha}f$ is absolutely continuous on the interval $[a,b]$, then the left
Riemann--Liouville fractional derivative is given by:
\begin{equation}\label{RLFD1}
{_aD_x^\alpha}f(x)=\frac{1}{\Gamma(1-\alpha)}\frac{d}{dx}\int_a^x
(x-t)^{-\alpha}f(t)dt=\frac{d}{dx}{_aI_x^{1-\alpha}}f(x).
\end{equation}
If the function ${_xI_b^{1-\alpha}}f$ is absolutely continuous on the interval $[a,b]$, then the right
Riemann--Liouville fractional derivative is given by:
\begin{equation}\label{RLFD2}
{_xD_b^\alpha}f(x)=\frac{-1}{\Gamma(1-\alpha)}\frac{d}{dx}\int_x^b
(t-x)^{-\alpha}
f(t)dt=\left(-\frac{d}{dx}\right){_xI_b^{1-\alpha}}f(x).
\end{equation}
Let $f\in AC([a,b],\R)$. By the left Caputo fractional derivative of $f$ on the interval $[a,b]$ we mean a function ${^C_aD_x^\alpha}f$ defined by:
\begin{equation}\label{CFD1}
{^C_aD_x^\alpha}f(x)=\frac{1}{\Gamma(1-\alpha)}\int_a^x
(x-t)^{-\alpha}\frac{d}{dt}f(t)dt={_aI_x^{1-\alpha}}\frac{d}{dx}f(x),
\end{equation}
and by the right Caputo fractional derivative of $f$ on the interval $[a,b]$ we mean a function ${^C_aD_x^\alpha}f$ defined by:
\begin{equation}\label{CFD2}
{^C_xD_b^\alpha}f(x)=\frac{-1}{\Gamma(1-\alpha)}\int_x^b
(t-x)^{-\alpha}
\frac{d}{dt}f(t)dt={_xI_b^{1-\alpha}}\left(-\frac{d}{dx}\right)f(x).
\end{equation}

\begin{remark}
Observe that if $\alpha$ goes to $1$, then under suitable assumptions operators ${_aD_x^\alpha}$ and ${^C_aD_x^\alpha}$ can be replaced with $\frac{d}{dx}$, and operators ${_{x}D_{b}^{{\alpha}}}$ and ${^C_xD_b^\alpha}$ can be replaced with $-\frac{d}{dx}$ (see \citet{book:Podlubny}).
\end{remark}

The operators \eqref{RLFI1}--\eqref{CFD2} are obviously linear. Below we present the rules of fractional integration by parts for Riemann--Liouville fractional integral and Caputo fractional derivatives which are particularly useful for our purposes.
\begin{lemma}[\citet{int:partsRef}]
Let $0<\alpha<1$, $p\geq1$,
$q \geq 1$, and $1/p+1/q\leq1+\alpha$. If $g\in L_p([a,b])$ and
$f\in L_q([a,b])$, then
\begin{equation}\label{ipi}
\int_{a}^{b}  g(x){_aI_x^\alpha}f(x)dx =\int_a^b f(x){_x I_b^\alpha}
g(x)dx.
\end{equation}
\end{lemma}

\begin{lemma}\label{byparts}[\textrm{cf.} \citet{book:Klimek}]
Let $0<\alpha<1$. If $f,g\in AC([a,b])$, then
\begin{equation}\label{ip}
\begin{split}
\int_{a}^{b}  g(x) \, {^C_aD_x^\alpha}f(x)dx &=\left.f(x){_x
I_b^{1-\alpha}} g(x)\right|^{x=b}_{x=a}+\int_a^b f(x){_x D_b^\alpha}
g(x)dx,\\
\int_{a}^{b}  g(x) \, {^C_xD_b^\alpha}f(x)dx &=\left.-f(x){_a
I_x^{1-\alpha}} g(x)\right|^{x=b}_{x=a}+\int_a^b f(x){_a D_x^\alpha}
g(x)dx.
\end{split}
\end{equation}
\end{lemma}

\begin{proof}
Formulas could be derived using equations \eqref{RLFD1}--\eqref{CFD2},
the identity \eqref{ipi} and performing standard integration by parts.
\end{proof}

Partial fractional integrals and derivatives are a natural generalization of the corresponding one-dimensional fractional integrals and derivatives, being taken with respect to one or several variables. For $(x_1,\ldots,x_n)$, $(\alpha_1,\ldots,\alpha_n)$, where $0<\alpha_i<1$, $i=1,\ldots,n$ and $[a_1,b_1]\times \ldots \times [a_n,b_n]$, the partial Riemann--Liouville fractional integrals of order $\alpha_k$ with respect to $x_k$ are defined by

\begin{equation*}
{_{a_{k}}I_{x_{k}}^{{\alpha}_k}}f(x_1,\ldots,x_n)=\frac{1}{\Gamma(\alpha_k)}\int_{a_k}^{x_k}(x_k-t_k)^{\alpha_k-1}f(x_1,\ldots,x_{k-1},t_k,x_{k+1},\ldots,x_n)dt_k,\quad x_k>a_k,
\end{equation*}
\begin{equation*}
{_{x_{k}}I_{b_{k}}^{{\alpha}_k}}f(x_1,\ldots,x_n)=\frac{1}{\Gamma(\alpha_k)}\int_{x_k}^{b_k}(t_k-x_k)^{\alpha_k-1}f(x_1,\ldots,x_{k-1},t_k,x_{k+1},\ldots,x_n)dt_k,\quad x_k<b_k.
\end{equation*}

Partial Riemann--Liouville and Caputo derivatives are defined by:
\begin{equation*}
{_{a_{k}}D_{x_{k}}^{{\alpha}_k}}f(x_1,\ldots,x_n)=\frac{1}{\Gamma(1-\alpha_k)}\frac{\partial}{\partial x_k}\int_{a_k}^{x_k}(x_k-t_k)^{-\alpha_k}f(x_1,\ldots,x_{k-1},t_k,x_{k+1},\ldots,x_n)dt_k,
\end{equation*}
\begin{equation*}
{_{x_{k}}D_{b_{k}}^{{\alpha}_k}}f(x_1,\ldots,x_n)=-\frac{1}{\Gamma(1-\alpha_k)}\frac{\partial}{\partial x_k}\int_{x_k}^{b_k}(t_k-x_k)^{-\alpha_k}f(x_1,\ldots,x_{k-1},t_k,x_{k+1},\ldots,x_n)dt_k,
\end{equation*}

\begin{equation*}
{^C_{a_{k}}D_{x_{k}}^{{\alpha}_k}}f(x_1,\ldots,x_n)=\frac{1}{\Gamma(1-\alpha_k)}\int_{a_k}^{x_k}(x_k-t_k)^{-\alpha_k}\frac{\partial}{\partial t_k}f(x_1,\ldots,x_{k-1},t_k,x_{k+1},\ldots,x_n)dt_k,
\end{equation*}
\begin{equation*}
{^C_{x_{k}}D_{b_{k}}^{{\alpha}_k}}f(x_1,\ldots,x_n)=-\frac{1}{\Gamma(1-\alpha_k)}\int_{x_k}^{b_k}(t_k-x_k)^{-\alpha_k}\frac{\partial}{\partial t_k}f(x_1,\ldots,x_{k-1},t_k,x_{k+1},\ldots,x_n)dt_k.
\end{equation*}

\section{Main Results}\label{main}

In this Section we formulate and prove the fractional second Noether type theorem, for single (Subsection~\ref{main:single}) and multiple (Subsection~\ref{main:multiple}) integral problems.

\subsection{Single Integral Case}\label{main:single}

Consider a system characterized by a set of functions
\begin{equation}\label{system:1}
x^i(t), \quad i=1,\ldots,n,
\end{equation}
depending on time $t$. We can simplify the notation by interpreting \eqref{system:1} as a vector function $x=(x^1,\ldots,x^n)$. Define the action functional in the form
\begin{equation}\label{funct1}
\mathcal{J}(x)=\int_a^b L(t,x(t),\, {_a^C D_t^\alpha} x(t))dt,
\end{equation}
where:
\begin{itemize}
\item[(i)] ${_a^C D_t^\alpha} x(t):=\left({_a^C D_t^{\alpha_1}} x^1(t),\ldots,{_a^C D_t^{\alpha_n}} x^n(t)\right)$, $0<\alpha_i\leq 1$, $i=1,\ldots, n$;
\item[(ii)] $x\in C^1([a,b],\R^n)$;
\item[(iii)] $L\in C^1([a,b]\times\mathbb{R}^{2n},\mathbb{R})$;
\item[(iv)] $t\rightarrow \frac{\partial L}{\partial{_a^C D_t^{\alpha_k}} x^k }\in AC([a,b])$ for every  $x\in C^1([a,b],\R^n)$, $k=1,\ldots,n$.
\end{itemize}

We define the admissible set of functions $A([a,b])$ by
$$A([a,b]):=\{x\in C^1([a,b],\R^n):x(a)=x_a,x(b)=x_b, x_a,x_b\in \R^n \}.$$

\begin{theorem}[\textbf{\citet{agrawalCap}}]
A necessary condition for the function $x\in A([a,b])$ to provide an extremum for the functional \eqref{funct1} is that its components satisfy the $n$ fractional equations
\begin{equation*}
\frac{\partial L}{\partial x^k}+{_t D_b^{\alpha_k}}
\frac{\partial L}{\partial {_a^C D_t^{\alpha_k}} x^k }=0,\quad k=1,\ldots, n
\end{equation*}
for $t\in[a,b].$
\end{theorem}
Define
$$E_k^f(L):=\frac{\partial L}{\partial x^k}+{_t D_b^{\alpha_k}}
\frac{\partial L}{\partial {_a^C D_t^{\alpha_k}} x^k }.$$
We shall call $E_k^f(L)$ the fractional Lagrange expressions. \\
The invariance transformations that we shall consider are infinitesimal transformations that depend upon arbitrary functions and
their fractional derivatives in the sense of Caputo. Let

\begin{equation}
\label{trans:frac}
\begin{cases}
\bar{t} = t ,\\
\bar{x}^k(t) = x^k(t) + T^{k1}(p_1(t))+\cdots+ T^{kr}(p_r(t)),\quad k=1,\ldots,n,\\
\end{cases}
\end{equation}

where $T^{ks}$ are linear fractional differential operators and $p_s$, $s=1,\ldots,r$ are $r$ arbitrary, independent $C^1$ functions defined on $[a,b]$. Then, we consider four types of fractional differential operators:
\begin{itemize}

\item [I.] Operator of the first kind
\begin{equation*}\label{op1}
T^{ks}=T^{ks}_1:=a_0^{ks}(t)+a_1^{ks}(t){_a^C D_t^{\beta_{ks1}}}+\cdots+a_l^{ks}(t){_a^C D_t^{\beta_{ksl}}},\quad 0< \beta_{ksi}\leq 1,
\end{equation*}
and ${_a^C D_t^{\beta_{ksi}}}p_s\in C^1([a,b])$, $a_i^{ks}\in C^1([a,b],\R)$, $s=1,\ldots,r$, $i=1,\ldots,l$. \\

\item [II.] Operator of the second kind
\begin{equation*}\label{op2}
T^{ks}=T^{ks}_2:=a_0^{ks}(t)+a_1^{ks}(t){^C_tD_b^{\beta_{ks1}}}+\cdots+a_l^{ks}(t){^C_tD_b^{\beta_{ksl}}},\quad 0< \beta_{ksi}\leq 1,
\end{equation*}
and ${_t^C D_b^{\beta_{ksi}}}p_s\in C^1([a,b])$, $a_i^{ks}\in C^1([a,b],\R)$, $s=1,\ldots,r$, $i=1,\ldots,l$.\\

\item [III.]  Operator of the third kind
\begin{equation*}\label{op3}
T^{ks}=T^{ks}_3:=a_0^{ks}(t)+a_1^{ks}(t){_a^C D_t^{\beta_{ks}}}+a_2^{ks}(t){_a^C D_t^{1+\beta_{ks}}}+\cdots+a_l^{ks}(t){_a^C D_t^{l-1+\beta_{ks}}},\quad 0< \beta_{ks}\leq 1,
\end{equation*}
$p_s\in C^l([a,b])$ and ${^C_aD_t^{1+\beta_{ks}}}p_s\in C^1([a,b])$, $a_i^{ks}\in C^1([a,b],\R)$, $s=1,\ldots,r$.\\

\item [IV.] Operator of the fourth kind
\begin{equation*}\label{op3}
T^{ks}=T^{ks}_4:=a_0^{ks}(t)+a_1^{ks}(t){^C_tD_b^{\beta_{ks}}}+a_1^{ks}(t){^C_tD_b^{1+\beta_{ks}}}+\cdots+a_l^{ks}(t){^C_tD_b^{l-1+\beta_{ks}}},\quad 0<\beta_{ks}\leq 1,
\end{equation*}
$p_s\in C^l([a,b])$ and ${^C_tD_b^{1+\beta_{ks}}}p_s\in C^1([a,b])$, $a_i^{ks}\in C^1([a,b],\R)$, $s=1,\ldots,r$.\\
\end{itemize}

Now we define the formal adjoint operator $\tilde{T}^{ks}$ of a fractional differential operator $T^{ks}$ similar in spirit to the classical case, by the integration by parts:
\begin{equation}\label{proof:3}
\int_a^bqT^{ks}(p_s)dt=\int_a^bp_s\tilde{T}^{ks}(q)dt+[\cdot]_{t=a}^{t=b},
\end{equation}
where $[\cdot]_{t=a}^{t=b}$ represents the boundary terms. Therefore, the adjoints of $T^{ks}_i$, $i=1,\ldots,4$, are given by expressions:

\begin{equation}\label{a:op:1}
\int_a^bqT_1^{ks}(p_s)dt=\int_a^bp_s\tilde{T}_1^{ks}(q)=\int_a^bp_s\left(a_0^{ks}q+\sum_{i=1}^{l}{_tD_b^{\beta_{ksi}}}(a_{i}^{ks}q)\right)dt+[\cdot]_{t=a}^{t=b},
\end{equation}

\begin{equation}\label{a:op:2}
\int_a^bqT_2^{ks}(p_s)dt=\int_a^bp_s\tilde{T}_2^{ks}(q)=\int_a^bp_s\left(a_0^{ks}q+\sum_{i=1}^{l}{_aD_t^{\beta_{ksi}}}(a_{i}^{ks}q)\right)dt+[\cdot]_{t=a}^{t=b},
\end{equation}

\begin{equation}\label{a:op:3}
\int_a^bqT_3^{ks}(p_s)dt=\int_a^bp_s\tilde{T}_3^{ks}(q)=\int_a^bp_s\left(a_0^{ks}q+\sum_{i=0}^{l-1}{_t D_b^{i+\beta_{ks}}}(a_{i+1}^{ks}q)\right)dt+[\cdot]_{t=a}^{t=b},
\end{equation}

\begin{equation}\label{a:op:4}
\int_a^bqT_4^{ks}(p_s)dt=\int_a^bp_s\tilde{T}_4^{ks}(q)=\int_a^bp_s\left(a_0^{ks}q+\sum_{i=0}^{l-1}{_a D_t^{i+\beta_{ks}}}(a_{i+1}^{ks}q)\right)dt+[\cdot]_{t=a}^{t=b}.
\end{equation}

Note that, if $p_s(a)=p_s(b)=0$, then  boundary terms in \eqref{a:op:1} and \eqref{a:op:2} vanish. Similarly, if  $p_s(a)=p^{(1)}_s(a)=\ldots=p^{(l-1)}_s(a)=0$ and $p_s(b)=p^{(1)}_s(b)=\ldots=p^{(l-1)}_s(b)=0$, then boundary terms in \eqref{a:op:3} and \eqref{a:op:4} vanish.

Now we define invariance.
\begin{definition}\label{invar:frac}
The functional \eqref{funct1} is invariant under transformations \eqref{trans:frac} if and only if for all $x\in C^1([a,b],\R^n)$ we have
$$\int_a^b L(t,\bar{x}(t),\, {_a^C D_t^\alpha} \bar{x}(t))dt=\int_a^b L(t,x(t),\, {_a^C D_t^\alpha} x(t))dt.$$
\end{definition}

\begin{theorem}
\label{theo:tnoe} If functional \eqref{funct1} is invariant under
transformations \eqref{trans:frac}, then there exist $r$ identities of the form
\begin{equation}\label{frac:iden}
\sum_{k=1}^{n}\tilde{T}^{ks}\left(E_k^f(L)\right)=0, \quad s=1,\ldots,r,
\end{equation}
where $\tilde{T}^{ks}$ is the adjoint of $T^{ks}$.
\end{theorem}

\begin{proof}
We give the proof only for the case $T^{ks}=T^{ks}_1$; other cases can be proved similarly.
By Definition~\ref{invar:frac} we have
\begin{multline*}
0=\int_a^b L(t,\bar{x}(t),\, {_a^C D_t^\alpha} \bar{x}(t))dt-\int_a^b L(t,x(t),\, {_a^C D_t^\alpha} x(t))dt\\
=\int_a^b \left(L(t,\bar{x}(t),\, {_a^C D_t^\alpha} \bar{x}(t))- L(t,x(t),\, {_a^C D_t^\alpha} x(t))\right)dt.
\end{multline*}
Then, by the Taylor formula
\begin{equation}\label{proof:1}
0=\sum_{k=1}^{n}\int_a^b\left[\frac{\partial L}{\partial x^k}T^{ks}_1(p_s)+\frac{\partial L}{\partial {_a^C D_t^{\alpha_k}} x^k }{_a^C D_t^{\alpha_k}}T^{ks}_1(p_s)\right]dt,
\end{equation}
where $T^{ks}_1(p_s)=\sum_{s=1}^{r}T^{ks}_1(p_s).$
The second term in the integrand may be integrated by parts (see the first formula of \eqref{ip}):
\begin{equation}\label{proof:2}
\int_a^b\frac{\partial L}{\partial {_a^C D_t^{\alpha_k}} x^k }{_a^C D_t^{\alpha_k}}T^{ks}_1(p_s)=\left.T^{ks}_1(p_s){_t
I_b^{1-{\alpha_{k}}}}\frac{\partial L}{\partial {_a^C D_t^{\alpha_k}} x^k }\right|^{x=b}_{x=a}+\int_a^b T^{ks}_1(p_s){_tD_b^{\alpha_k}}
\frac{\partial L}{\partial {_a^C D_t^{\alpha_k}} x^k }dt.
\end{equation}
Since $p_s$ are arbitrary, we may choose $p_s$ such that: $p_s(a)=p_s(b)=0$ and ${_a^C D_t^{\beta_{ksi}}}p_s(t)|_{t=b}=0$, $s=1,\ldots,r$, $i=1,\ldots,l$; and if $\beta_{ksi}=1$, then also ${_a^C D_t^{\beta_{ksi}}}p_s(t)|_{t=a}=0$. Therefore, the boundary term in \eqref{proof:2} vanishes and substituting \eqref{proof:2} into \eqref{proof:1} we get
\begin{equation*}\label{proof:4}
0=\sum_{k=1}^{n}\int_a^b\left[\frac{\partial L}{\partial x^k}+{_t D_b^{\alpha_k}}
\frac{\partial L}{\partial {_a^C D_t^{\alpha_k}} x^k }\right]T^{ks}_1(p_s)dt.
\end{equation*}
Using the definition of the adjoint operator $\tilde{T}^{ks}_1$ of a fractional differential operator $T^{ks}_1$, that is, equation \eqref{a:op:1}, we get \begin{equation*}
0=\sum_{k=1}^{n}\int_a^b\sum_{s=1}^{r}\tilde{T}^{ks}_{1}\left(E_k^f(L)\right)p_sdt+[\cdot]_{t=a}^{t=b}.
\end{equation*}
Again appealing to the arbitrariness of $p_s$ we can force the boundary term to vanish, and
finally by the fundamental lemma of calculus of variations we conclude that
\begin{equation*}
\sum_{k=1}^{n}\tilde{T}^{ks}_{1}\left(E_k^f(L)\right)=0, \quad s=1,\ldots,r.
\end{equation*}
\end{proof}

\begin{remark}
Notice that if we put $\beta_{ks}=1$ in transformations of the third or the fourth kind, then we obtain infinitesimal transformations:
\begin{equation*}
\begin{cases}
\bar{t} = t\\
\bar{x}^k(t) = x^k + B^{ks}(p_s) + \ldots\, ,\\
\end{cases}
\end{equation*}
where
\begin{equation*}
B^{ks}=b_0^{ks}(t)+b_1^{ks}(t)\frac{d}{dt}+b_2^{ks}(t)\frac{d^2}{dt^2}+\cdots+b_l^{ks}(t)\frac{d^l}{dt^l},\quad k=1,\ldots,n.
\end{equation*}
In this case the adjoint operator $\tilde{B}^{ks}$ of the differential operator $B^{ks}$ is given by
\begin{equation*}
\tilde{B}^{ks}(q)=b_0^{ks}q+\sum_{i=1}^{l}(-1)^i\frac{d^i}{dt^i}(b_{i}^{ks}q),\quad k=1,\ldots,n
\end{equation*}
and the identities \eqref{frac:iden} take the form
\begin{equation*}
\sum_{k=1}^{n}b_0^{ks}(E_k(L))+\sum_{k=1}^{n}\sum_{i=1}^{l}(-1)^i\frac{d^i}{dt^i}\left(b_{i}^{ks}E_k(L)\right)=0,\quad s=1,\ldots,r,
\end{equation*}
which are exactly the Noether identities (see \citet{brading1,logan1}).
\end{remark}

\begin{remark}
The fractional differential operators $T^{ks}_1$, $T^{ks}_2$, $T^{ks}_3$ and $T^{ks}_4$ can of course be combined, that is, we can consider infinitesimal transformations that depend upon arbitrary functions and their fractional derivatives in the sense of Caputo: left and right with various orders.
\end{remark}

\subsection{Multiple Integral Case}\label{main:multiple}

Consider a system characterized by a set of functions
\begin{equation}\label{system}
u^j(t,x_1,\ldots,x_m), \quad j=1,\ldots,n,
\end{equation}
depending on time $t$ and the space coordinates $x_1,\ldots,x_m$. We can simplify the notation by interpreting \eqref{system} as a vector function $u=(u^1,\ldots,u^n)$ and writing $t=x_0$, $x=(x_0,x_1,\ldots,x_m)$, $dx=dx_0dx_1\cdots dx_m$. Then \eqref{system} becomes simply $u(x)$ and is called a vector field. Define the action functional in the form
\begin{equation}\label{problem}
\mathcal{J}(u)=\int_{\Omega}\mathcal{L}(x,u,{^{C}\nabla^{\alpha}}u)dx,
\end{equation}
where $\Omega=R\times [a_0,b_0]$, $R=[a_1,b_1]\times \ldots \times [a_m,b_m]$, and ${^{C}\nabla^{\alpha}}$ is the operator
\begin{equation*}
({^{C}_{a_0}D_{x_{0}}^{{\alpha}_0}},{^{C}_{a_{1}}D_{x_{1}}^{{\alpha}_1}},\cdots, {^{C}_{a_{m}}D_{x_{m}}^{{\alpha}_m}}),
\end{equation*}
where $\alpha=(\alpha_0,\alpha_1,\ldots,\alpha_m)$, $0<\alpha_i\leq1$, $i=0,\ldots,m$. The function $\mathcal{L}(x,u,{^{C}\nabla^{\alpha}}u)$ is called the fractional Lagrangian density of the field.
We assume that:
\begin{itemize}
\item[(i)] $u^j\in C^1(\Omega,\mathbb{R})$,  $j=1,\ldots, n$;
\item[(ii)]
$\mathcal{L}\in C^1(\mathbb{R}^m\times\mathbb{R}^n\times\mathbb{R}^{n(m+1)}; \mathbb{R})$;
\item[(iii)] $x \rightarrow \frac{\partial \mathcal{L}}{\partial {^{C}_{a_{i}}D_{x_{i}}^{{\alpha}_i}}u^j}$ for every $u^j\in C^1(\Omega,\mathbb{R})$ are $C^1$ functions, $j=1,\ldots,n$, $i=0,\ldots,m$.
\end{itemize}

Define the admissible set of functions $A(\Omega)$ by
$$A(\Omega):=\{u: \Omega \rightarrow \mathbb{R}^n:u(x)=\varphi(x)\quad \mbox{for}\quad x\in \partial \Omega\},$$
where $\varphi:\partial \Omega \rightarrow \mathbb{R}^{n}$ is a given function.

Applying the principle of stationary action to \eqref{problem} we obtain the multidimensional fractional Euler--Lagrange equations for the field.
\begin{theorem}\label{th:E-L:m}[\textbf{\textrm{cf.} \citet{Cresson}}]
A necessary condition for the function $u\in A(\Omega)$ to provide an extremum for the action functional \eqref{problem} it that its components satisfy the $n$ multidimensional fractional Euler--Lagrange equations:
\begin{equation*}\label{E-L}
\frac{\partial \mathcal{L}}{\partial u^j}+\sum_{i=0}^m {_{x_i}D_{b_{i}}^{{\alpha}_i}}\frac{\partial \mathcal{L}}{\partial {^{C}_{a_{i}}D_{x_{i}}^{{\alpha}_i}}u^j}=0,\quad j=1,\ldots,n.
\end{equation*}
\end{theorem}

As before we define
$$E_j^f(\mathcal{L}):=\frac{\partial \mathcal{L}}{\partial u^j}+\sum_{i=0}^n {_{x_i}D_{b_{i}}^{{\alpha}_i}}\frac{\partial \mathcal{L}}{\partial {^{C}_{a_{i}}D_{x_{i}}^{{\alpha}_i}}u^j},$$
which are called the fractional Lagrange expressions. \\
We shall study infinitesimal transformations that depend upon arbitrary functions of independent variables and their partial fractional derivatives in the sense of Caputo. Let

\begin{equation}
\label{trans:frac:m}
\begin{cases}
\bar{x} = x ,\\
\bar{u}^j(x) = u^j(x) + T^{j1}(p_1(x))+\cdots+ T^{jr}(p_r(x)),\quad j=1,\ldots,n,\\
\end{cases}
\end{equation}

where $T^{js}$ are linear fractional differential operators and $p_s$, $s=1,\ldots,r$ are the $r$ arbitrary, independent $C^1$ functions defined on $\Omega$. Then, we consider two types of fractional differential operators:
\begin{itemize}

\item [I.] Operator of the first kind
\begin{equation*}\label{op1:m}
T^{js}=T^{js}_1:=c^{js}(x)+\sum_{i=0}^{m}c_i^{js}(x){_{a_i}^C D_{x_i}^{\beta_{jsi}}},\quad 0< \beta_{jsi}\leq 1,
\end{equation*}
and ${_{a_i}^C D_{x_i}^{\beta_{jsi}}}p_s$, $c^{js}$, $c_i^{js}$ are $C^1$ functions defined on $\Omega$, $s=1,\ldots,r$, $i=1,\ldots,m$.

\item [II.] Operator of the second kind
\begin{equation*}\label{op2:m}
T^{js}=T^{js}_1:=c^{js}(x)+\sum_{i=0}^{m}c_i^{js}(x){_{x_i}^C D_{b_i}^{\beta_{jsi}}},\quad 0< \beta_{jsi}\leq 1,
\end{equation*}
and ${_{x_i}^C D_{b_i}^{\beta_{jsi}}}p$, $c^{js}$, $c_i^{js}$ are $C^1$ functions defined on $\Omega$, $s=1,\ldots,r$, $i=1,\ldots,m$.

\end{itemize}

We define invariance similarly to the one-dimensional case.
\begin{definition}\label{invar:frac:m}
The functional \eqref{problem} is invariant under transformations \eqref{trans:frac:m} if and only if for all $u\in C^1(\Omega,\mathbb{R}^n)$ we have
$$\int_{\Omega}\mathcal{L}(x,\bar{u},{^{C}\nabla^{\alpha}}\bar{u})dx=\int_{\Omega}\mathcal{L}(x,u,{^{C}\nabla^{\alpha}}u)dx.$$
\end{definition}

\begin{theorem}
\label{theo:tnoe:m} If functional \eqref{problem} is invariant under
transformations \eqref{trans:frac:m}, then there exist $r$ identities of the form
\begin{equation*}
\sum_{j=1}^{n}\tilde{T}^{js}\left(E_j^f(\mathcal{L})\right)=0, \quad s=1,\ldots,r,
\end{equation*}
where $\tilde{T}^{js}$ is the adjoint of $T^{js}$.
\end{theorem}

\begin{proof}
We give the proof only for the case $T^{js}=T^{js}_1$; the other case can be proved similarly.
By Definition~\ref{invar:frac:m} we have
\begin{equation*}
0=\int_{\Omega}\mathcal{L}(x,\bar{u},{^{C}\nabla^{\alpha}}\bar{u})dx-\int_{\Omega}\mathcal{L}(x,u,{^{C}\nabla^{\alpha}}u)dx
=\int_{\Omega}\left(\mathcal{L}(x,\bar{u},{^{C}\nabla^{\alpha}}\bar{u})-\mathcal{L}(x,u,{^{C}\nabla^{\alpha}}u)\right)dx.
\end{equation*}
Then, by the Taylor formula
\begin{equation}\label{proof:1:m}
0=\sum_{j=1}^{n}\int_{\Omega}\left(\frac{\partial \mathcal{L}}{\partial u^j}T^{js}_1(p_s)+\sum_{i=0}^m \frac{\partial \mathcal{L}}{\partial {^{C}_{a_{i}}D_{x_{i}}^{{\alpha}_i}}u^j}{^{C}_{a_{i}}D_{x_{i}}^{{\alpha}_i}}T^{js}_1(p_s) \right)dx,
\end{equation}
where $T^{js}_1(p_s)=\sum_{s=1}^{r}T^{js}_1(p_s).$
The Fubini theorem allows us to rewrite integrals as the iterated integrals so that we can use the integration by parts formula \eqref{ip}:
\begin{equation}\label{proof:2:m}
\int_{\Omega}\sum_{i=0}^m \frac{\partial \mathcal{L}}{\partial {^{C}_{a_{i}}D_{x_{i}}^{{\alpha}_i}}u^j}{^{C}_{a_{i}}D_{x_{i}}^{{\alpha}_i}}T^{js}_1(p_s) dx=\int_{\Omega}\sum_{i=0}^m {_{x_{i}}D_{b_{i}}^{{\alpha}_i}}\frac{\partial \mathcal{L}}{\partial {^{C}_{a_{i}}D_{x_{i}}^{{\alpha}_i}}u^j}T^{js}_1(p_s)dx+[\cdot]|_{\partial \Omega},\quad j=1,\ldots,m,
\end{equation}
where $[\cdot]|_{\partial \Omega}$ represent the boundary terms -- $m$-volumes integrals. Since $p_s$ are arbitrary, we may choose $p_s$ such that: $p_s(x)|_{\partial \Omega}=0$ and ${_{a_i}^C D_{x_i}^{\beta_{jsi}}}p_s(t)|_{\partial \Omega}=0$, $s=1,\ldots,r$, $i=1,\ldots,l$. Therefore, the boundary term in \eqref{proof:2:m} vanishes and substituting \eqref{proof:2:m} into \eqref{proof:1:m} we get
\begin{equation*}\label{proof:4:m}
0=\sum_{j=1}^{n}\int_{\Omega}\left(\frac{\partial \mathcal{L}}{\partial u^j}+\sum_{i=0}^m {_{x_{i}}D_{b_{i}}^{{\alpha}_i}}\frac{\partial \mathcal{L}}{\partial {^{C}_{a_{i}}D_{x_{i}}^{{\alpha}_i}}u^j} \right)T^{js}_1(p_s)dx.
\end{equation*}
Now we proceed as in the one-dimensional case and define the adjoint operator $\tilde{T}^{js}_1$ of a fractional differential operator $T^{js}_1$ by \begin{equation*}
\int_{\Omega}q(x)T^{js}_1(p_s(x))dx=\int_{\Omega}p_s(x)\tilde{T}^{js}_1(q(x))dx+[\cdot]|_{\partial \Omega},\quad j=1,\ldots,n,\quad s=1,\ldots,r,
\end{equation*}
where we use the Fubini theorem. Again appealing to the arbitrariness of $p_s$ we can force the boundary term to vanish (by putting $p_s(x)|_{\partial \Omega}=0$). Therefore,
$$0=\sum_{j=1}^{n}\int_{\Omega}\sum_{s=1}^{r}\tilde{T}^{js}_1\left(\frac{\partial \mathcal{L}}{\partial u^j}+\sum_{i=0}^m {_{x_{i}}D_{b_{i}}^{{\alpha}_i}}\frac{\partial \mathcal{L}}{\partial {^{C}_{a_{i}}D_{x_{i}}^{{\alpha}_i}}u^j} \right)p_sdx.$$
Finally by the fundamental lemma of calculus of variations we conclude that
\begin{equation*}
\sum_{j=1}^{n}\tilde{T}^{js}_{1}\left(E_j^f(\mathcal{L})\right)=0, \quad s=1,\ldots,r.
\end{equation*}
\end{proof}

\begin{remark}
The adjoints of $T^{js}_i$, $i=1,2$, are given by expressions:
\begin{equation*}\label{a:op:1:m}
\tilde{T}_1^{js}(q)=c^{js}q+\sum_{i=0}^{m}{_{x_i} D_{b_i}^{\beta_{jsi}}}(c_i^{js}q),\quad j=1,\ldots,n,
\end{equation*}
\begin{equation*}\label{a:op:2:m}
\tilde{T}_2^{js}(q)=c^{js}q+\sum_{i=0}^{m}{_{a_i} D_{x_i}^{\beta_{jsi}}}(c_i^{js}q), \quad j=1,\ldots,n.
\end{equation*}
\end{remark}

\begin{remark}
The fractional differential operators $T^{js}_1$ and $T^{js}_2$ can of course be combined, that is, we can consider infinitesimal transformations that depend upon arbitrary functions and their partial fractional derivatives in the sense of Caputo: left and right with various orders.
\end{remark}

\section{Example}\label{example}

In order to illustrate our result we will use the Lagrangian density for the electromagnetic field (see \citet{gelfand}):
\begin{equation}\label{l:em}
\mathcal{L}=\frac{1}{8\pi}(\mathbf{E}^2-\mathbf{H}^2),
\end{equation}
where $\mathbf{E}$ and $\mathbf{H}$ are the electric field vector and the magnetic field vector, respectively. Following (\citet{Baleanu3}) we shall generalize \eqref{l:em} to the fractional Lagrangian density by changing classical partial derivatives by fractional. Let $x=(x_0,x_1,x_2,x_3)\in \Omega$; and $\mathbf{A}(x)=(A_1(x),A_2(x),A_3(x))$, $A_0(x)$ be a vector potential and a scalar potential, respectively. They are defined by setting
\begin{equation}\label{p:em}
\mathbf{E}=grad^{(\alpha_1,\alpha_2,\alpha_3)}A_0-{_{a_0} ^CD_{x_0}^{\alpha_{0}}}\mathbf{A},\quad \mathbf{H}=curl\mathbf{A},\quad 0<\alpha_{i}\leq1,\quad i=0,\ldots,3,
\end{equation}
where
\begin{equation*}
grad^{(\alpha_1,\alpha_2,\alpha_3)}A_0=\mathbf{i} {_{a_1} ^CD_{x_1}^{\alpha_{1}}}A_0+\mathbf{j}{_{a_2} ^CD_{x_2}^{\alpha_{2}}}A_0+\mathbf{k}{_{a_3} ^CD_{x_3}^{\alpha_{3}}}A_0,
\end{equation*}
\begin{equation*}
{_{a_0} ^CD_{x_0}^{\alpha_{0}}}\mathbf{A}=\mathbf{i} {_{a_0} ^CD_{x_0}^{\alpha_{0}}}A_1+\mathbf{j}{_{a_0} ^CD_{x_0}^{\alpha_{0}}}A_2+\mathbf{k}{_{a_0} ^CD_{x_0}^{\alpha_{0}}}A_3,
\end{equation*}
\begin{equation*}
curl\mathbf{A}=\mathbf{i}({_{a_2} ^CD_{x_2}^{\alpha_{2}}}A_3-{_{a_3} ^CD_{x_3}^{\alpha_{3}}}A_2)+\mathbf{j}({_{a_3} ^CD_{x_3}^{\alpha_{3}}}A_1-{_{a_1} ^CD_{x_1}^{\alpha_{1}}}A_3)+\mathbf{k}({_{a_1} ^CD_{x_1}^{\alpha_{1}}}A_2-{_{a_2} ^CD_{x_2}^{\alpha_{2}}}A_1).
\end{equation*}
Replacing $\mathbf{E}$ and $\mathbf{H}$ in \eqref{l:em} by their expressions \eqref{p:em} we obtain the fractional Lagrangian density
\begin{equation}\label{l:em:f}
\mathcal{L}=\frac{1}{8\pi}\left[(grad^{(\alpha_1,\alpha_2,\alpha_3)}A_0-{_{a_0}^CD_{x_0}^{\alpha_{0}}}\mathbf{A})^2-(curl\mathbf{A})^2\right].
\end{equation}
Note that, similarly to the integer case, the potential $(A_0,\mathbf{A})$ is not uniquely determined by the vectors $\mathbf{E}$ and $\mathbf{H}$. Namely, $\mathbf{E}$ and $\mathbf{H}$ do not change if we make a gauge transformation:
\begin{equation}\label{g:em}
\tilde{A}_j(x)=A_j(x)+{_{a_j} ^CD_{x_j}^{\alpha_{j}}}f(x),\quad j=0,\ldots,3,
\end{equation}
where $f:\Omega\rightarrow \R$ is an arbitrary function of class $C^2$ in all of its argument. Therefore, the Lagrangian density \eqref{l:em:f}, and hence the action functional, is invariant under transformation \eqref{g:em}. By Theorem~\eqref{theo:tnoe:m}, we conclude that
\begin{equation*}
\sum_{j=0}^{3}{_{x_j} D_{b_j}^{\alpha_{j}}}\left(E_j^f(\mathcal{L})\right)=0,
\end{equation*}
where $E_j^f(\mathcal{L})$ are Lagrange expressions corresponding to \eqref{l:em:f}. Equations $E_j^f(\mathcal{L})=0$ do not uniquely determine the potential $(A_0,\mathbf{A})$  and to avoid this lack of uniqueness, the fractional Lorentz condition can be imposed on $(A_0,\mathbf{A})$.

\section*{Acknowledgments}

The author is supported
by the Bia{\l}ystok University of Technology Grant S/WI/2/2011 and by the European Union Human Capital Programme: Podniesienie potencja{\l}u uczelni wyzszych jako czynnik rozwoju gospodarki opartej na wiedzy. The author is grateful to two anonymous reviewers for their comments.


\end{document}